\documentclass[11pt,a4paper]{amsart}
\usepackage{latexsym,amssymb,amsmath}

\usepackage{enumerate}

\usepackage{amsmath}
\usepackage{amsfonts}
\usepackage{amssymb}
\usepackage{eucal}
\usepackage{hyperref}

\usepackage{amssymb}

\usepackage[shortalphabetic]{amsrefs}

\newcommand{\Shigh}{\mbox{\rm \textsf{Shigh}}}
\newcommand{\HH}{\mathcal{K}}
\newcommand{\cost}{c_{\HH}}

\newcommand{\exo}[1]{\exists #1 \, }
\newcommand{\fao}[1]{\forall #1 \, }
\newcommand{\GL}[1]{\mbox{\rm GL}_{#1}}

\newcommand{\ltt}{\le_{\mathrm{tt}}}

\newcommand{\ttext}[1]{\ \text{#1}\ }

\newcommand{\SI}[1]{\Sigma^0_{#1}}
\newcommand{\PI}[1]{\Pi^0_{#1}}
\newcommand{\PPI}{\PI{1}}

\newcommand{\Kuc}{Ku{\v c}era}

\newcommand{\strcantor}{2^{<\omega}}

\newcommand{\MLR}{\mbox{\rm \textsf{MLR}}}

\newcommand{\BLR}{\mbox{\rm \textsf{BLR}}}

\newcommand{\NN}{{\mathbb{N}}}

\newcommand{\ex}{\exists}
\newcommand{\fa}{\forall}

\newcommand{\n}{\noindent}
\newcommand{\wt}{\widetilde}

\newcommand{\sub}{\subseteq}

\renewcommand{\land}{\&}

\newcommand{\ES}{\emptyset}

\renewcommand{\tilde}{\widetilde}

\newcommand{\lland}{\ \land \ }

\newcommand{\la}{\langle}
\newcommand{\ra}{\rangle}

\newcommand{\bi}{\begin{itemize}}
\newcommand{\ei}{\end{itemize}}
\newcommand{\bc}{\begin{center}}
\newcommand{\ec}{\end{center}}

\newcommand{\sss}{\sigma}
\newcommand{\aaa}{\alpha}

\newcommand{\leT}{\le_T}

\newcommand{\leb}{\mathbf{\lambda}}

\newcommand{\Om}{\Omega}
\newcommand{\twoset}{\{0,1\}}

\newcommand{\sN}[1]{_{#1\in \NN}}
\newcommand{\estring}{\emptyset}

\newcommand{\tp}[1]{2^{#1}}

\newcommand{\uhr}[1]{\!  \upharpoonright_{#1}}

\newtheorem{theorem}{Theorem}[section]
\newtheorem{thm}[theorem]{Theorem}

\newtheorem{definability lemma}[theorem]{Definability Lemma}

\newtheorem{fact}[theorem]{Fact}

\newtheorem{df}[theorem]{Definition}

\newtheorem{lemma}[theorem]{Lemma}

\newtheorem{cor}[theorem]{Corollary}
\newtheorem{question}[theorem]{Question}
\newcommand{\N}{\mathbb N}
\newcommand{\Implies}{\Rightarrow}
\begin{document}

\title{Superhighness}
\author{Bj{\o}rn Kjos-Hanssen}
\address{Department of Mathematics\\
University of Hawaii at Manoa\\ 2565 McCarthy Mall \\ Honolulu, HI 96822 \\ USA} \email{bjoern@math.hawaii.edu}

\author{Andr\'e Nies}
\address{Department of Computer Science\\
  University of Auckland\\
  Private Bag 92019\\ Auckland \\ New Zealand } \email{andre@cs.auckland.ac.nz}

\begin{abstract}
We prove that   superhigh sets  can be jump traceable, answering a question of Cole and Simpson. On the other hand, we show that such sets cannot be weakly 2-random. We also study the class superhigh$^\Diamond$, and show that it contains some, but not all, of the noncomputable $K$-trivial sets. 
\end{abstract}

\thanks{Kjos-Hanssen is partially supported under NSF grant DMS-0652669. Nies is supported under Marsden fund UOA-08-187.}

\maketitle

\section{Introduction} 

An important non-computable set of integers in computability theory is $\ES'$, the halting problem for Turing machines. Over the last half century many interesting results have been obtained about   ways in which a problem can be almost as hard as $\ES'$. The \emph{superhigh} sets  are the sets $A$ such that $$A'\ge_{tt}\ES'',$$ i.e., the halting problem relative to $A$ computes $\ES''$ using a truth-table  reduction. The name comes from comparison with the \emph{high} sets, where instead arbitrary Turing reductions are allowed ($A'\ge_T \ES''$). Superhighness for computably enumerable (c.e.) sets was introduced by Mohrherr \cite{Mohrherr}. She proved that  the superhigh c.e.\ degrees sit properly between the high and Turing complete ($A\ge_T \ES'$) ones. 

Most questions one can ask on superhighness are currently open. For instance, Martin \cite{Martin} (1966) famously proved that a degree is high iff it can compute a function dominating all computable functions, but it is not known whether superhighness can be characterized in terms of domination. Cooper \cite{Cooper} showed that there is a high minimal Turing degree, but we do not know whether a superhigh set can be of minimal Turing degree. We hope the present paper lays the groundwork for a future understanding of these problems.

We prove that a superhigh set can be jump traceable. Let  superhigh$^\Diamond$ be the class of c.e.\ sets Turing below all Martin-L\"of random (ML-random) superhigh sets (see \cite[Section 8.5]{Nies:book}). We show that this class contains a promptly simple set,  and  is a proper subclass of the c.e. $K$-trivial sets.  This class was recently shown to coincide with the strongly jump traceable c.e.\  sets,    improving our result \cite{Nies:nd}.

\begin{df}
Let $\{\Phi_n^X\}_{n\in\N}$ denote a standard list of all functions partial computable in $X$, and let $W_n^X$ denote the domain of $\Phi_n^X$. We write $J^X(n)$ for $\Phi^X_n(n)$, and $J^\sss(n)$ for $\Phi^\sss_n(n)$ where $\sss$ is a string. Thus $X'=\{e\colon \ J^B(e)\downarrow\}$ represents the halting problem relative to $X$. 

 $X$ is \emph{jump-traceable by} $Y$ (written $X \le_{JT} Y$) if there exist computable functions $f(n)$ and $g(n)$ such that for all $n$, if $J^X(n)$ is defined ($J^X(n)\downarrow$) then $J^X(n)\in W_{f(n)}^Y$ and for all $n$, $W_{f(n)}^Y$ is finite of cardinality $\le g(n)$.
\end{df} The relation $\le_{JT}$ is transitive and indeed a weak reducibility \cite[8.4.14]{Nies:book}.  Further information on weak reducibilities, and jump traceability, may be found in the  recent book by Nies \cite{Nies:book}, especially in Sections 5.6 and 8.6, and 8.4, respectively.

\begin{df} $A$ is \emph{JT-hard} if $\ES'$ is jump traceable by $A$.
Let $\Shigh =\{Y: Y'\ge_{tt} \ES''\}$ be the class of superhigh sets.
\end{df}

\begin{thm} \label{consider}
Consider the following five properties of a set $A$.
\begin{enumerate}
\item $A$ is Turing complete; 

\item $A$ is almost everywhere dominating; 

\item $A$ is JT-hard; 
 
\item $A$ is superhigh;

\item $A$ is high. 
\end{enumerate}

We have (1)$\Implies$(2)$\Implies$(3)$\Implies$(4)$\Implies$(5), all implications being strict. 
\end{thm}
\begin{proof} Implications:
(1)$\Implies$(2): Dobrinen and Simpson \cite{DS}.
(2)$\Implies$(3): Simpson \cite{SimpsonSuper} Lemma 8.4.
(3)$\Implies$(4): Simpson \cite{SimpsonSuper} Lemma 8.6.
(4)$\Implies$(5): Trivial, since each truth-table reduction is a Turing reduction.

Non-implications: (2)$\not\Implies$(1) was proved by Cholak, Greenberg, and Miller \cite{CGM}.
(3)$\not\Implies$(2): By Cole and Simpson \cite{CS}, (3) coincides with (4) on the $\Delta^0_2$ sets. But there is a superhigh degree that does not satisfy (2): one can use Jockusch-Shore Jump Inversion for a super-low but not $K$-trivial set, which exists by the closure of the $K$-trivials under join and the existence of a pair of super-low degrees joining to $\ES'$. 
(4)$\not\Implies$(3): We prove in Theorem \ref{kjosEmail} below that there is a jump traceable superhigh degree. By transitivity of $\le_{JT}$ and the observation that $\ES'\not\le_{JT}\ES$, no jump traceable degree is JT-hard.
(5)$\not\Implies$(4): Binns, Kjos-Hanssen, Lerman, and Solomon \cite{BKLS} proved this using a syntactic analysis combined with a result of Schwartz \cite{Schwartz}.
\end{proof}

Historically, the easiest separation (1)(5) is a corollary of Friedberg's Jump Inversion Theorem \cite{Friedberg} from 1957. The separation (1)(4) follows similarly from Mohrherr's Jump Inversion Theorem for the tt-degrees \cite{Mohrherr} (1984), and the separation (4)(5) is essentially due to Schwartz \cite{Schwartz} (1982). The classes (2) and (3) were introduced more recently, by Dobrinen and Simpson \cite{DS} (2004) and Simpson \cite{SimpsonSuper} (2007). 

Notion (3), JT-hardness, may not appear to be very natural. However,  Cole and Simpson \cite{CS} gave  an embedding of the hyperarithmetic hierarchy $\{0^{(\alpha)}\}_{\alpha<\omega_1^{CK}}$ into the lattice of $\Pi^0_1$ classes under Muchnik reducibility making use of the notion of  \emph{bounded limit recursive} (\BLR) functions. We will see that JT-hardness coincides with  \BLR-hardness.
 
{\it Notation. }  We write \bc $\forall n \, f(n) = \lim_s^\text{\rm comp} \wt  f(n,s) $  \ec 
 if  for all $n$, $f(n)=\lim_s \tilde f(n, s)$, and moreover there is  a computable   function $g: \omega\rightarrow\omega$ such that for all $n$,
$\{s\, | \,\tilde f(n, s) \ne \tilde f(n, s+1)\}$ has  cardinality less than $g (n)$.

\section{Superhighness and jump traceability}
In this section we show that superhighness is compatible with the lowness property of being jump traceable, and deduce an answer to a question of Cole and Simpson. 
\begin{thm}\label{kjosEmail}
There is a superhigh jump-traceable set.
\end{thm}
\begin{proof}
Mohrherr \cite{Mohrherr}
proves a jump inversion theorem in the tt-degrees: For each set $A$, if $\ES'\le_{tt} A$, then there
exists a set $B$ such that $B'\equiv_{tt} A$.
To produce $B$, Mohrherr uses the same construction as in the proof of Friedberg's Jump Inversion Theorem for the Turing degrees. Namely, $B$ is constructed by finite extensions $B[s]\preceq B[s+1]\preceq\cdots$ Here $B[s]$ is a finite binary string and $\sigma\preceq\tau$ denotes that $\sigma$ is an initial substring of $\tau$. At stages of the form $s=2e$ (even stages), one searches for an extension $B[s+1]$ of $B[s]$ such that $J^{B[s+1]}(e)\downarrow$. If none is found one lets $B[s+1]=B[s]$. At stages of the form $s=2e+1$ (odd stages) one appends the bit $A(e)$, i.e. one lets $B[s+1]=B[s]^\frown\la A(e)\ra$. Thus  two types of oracle questions are asked alternately for varying numbers $e$: 

\begin{enumerate}
\item[(1)] Does a string $\sigma\succeq B[s]$ exist so that $J^\sigma(e)\downarrow$, i.e. $B\succeq \sigma$ implies $e\in B'$? (If so, let $B[s+1]$ be the first such string that is found.)

\item[(2)] Is $A(e)=1$?
\end{enumerate}
This allows for a jump trace $V_e$ of size at most $4^{e}$. First, $V_0$ consists of at most one value, namely the first value $J^\sigma(e)$ found for any $\sigma$ extending the empty string. Next, $V_1$ consists of the first value for $\Phi^\tau_1(1)$ found for any $\tau$ extending $\la 0\ra$, $\la 1\ra$, $\sigma^\frown \la 0\ra$, $\sigma^\frown\la 1\ra$, respectively, in the cases: $0\not\in A$, and $0\not\in B'$; $0\in A$ and $0\not\in B'$; $0\not\in A$ and $0\in B'$; and $0\in A$ and $0\in B'$.
 Generally, for each $e$ there are four possibilities: either $e$ is in $A$ or not, and either the extension $\sigma$ of $B[s]$ is found or not.
$V_e$ consists of all the possible values of $J^B(e)$ depending on the answers to these questions.

Hence $B$ is jump traceable, no matter what oracle $A$ is used. Thus,  letting $A=\ES''$ results in  a superhigh jump-traceable set $B$.
\end{proof}

\begin{question}
Is there  a superhigh set of  minimal Turing degree?
\end{question}

This question is sharp in terms of the notions (1)--(5) of Theorem \ref{consider}: minimal Turing degrees can be high (Cooper \cite{Cooper}) but not JT-hard (Barmpalias \cite{B}). 

Cole and Simpson \cite{CS} introduced the following notion. Let $A$ be a Turing oracle. A function $f \colon \,  \omega\rightarrow\omega$ is \emph{boundedly
limit computable by  $A$} if there exist an $A$-computable   function $\tilde f :
\omega\times\omega\rightarrow\omega$ such that  $\lim_s^{\text{\rm comp}} \tilde f(n, s) = f(n)$.

We write
\[\BLR(A) = \{f \in \omega^\omega\,  |\, f \text{ is boundedly limit computable by  }A\}.\]
We say that  $X \le_{BLR} Y $ if $\BLR(X)\subseteq \BLR(Y)$. In particular,   $A$ is \emph{\BLR-hard} if $\BLR(\ES')\subseteq \BLR(A)$.

It is easy to see that $\le_{BLR}$ implies $\le_{JT}$ (Lemma 6.8 of Cole and Simpson \cite{CS}). The following partial converse  is implicit in some recent papers  as pointed out to the   authors by Simpson.

\begin{thm}  Suppose that  $A\le_{JT}B$ where $A$ is a c.e.\ set and $B$ is any set.  Then \BLR$(A)\subseteq$\BLR$(B)$.
\end{thm}
\begin{proof}
Since $A \le_{JT} B$, by Remark 8.7 of Simpson \cite{SimpsonSuper}, the function $h$ given by 
\bc $h(e)=J^A(e)+1$ if $J^A(e)\downarrow$, $h(e)=0$ otherwise, \ec
is $B'$-computable, with computably bounded use of $B'$ and unbounded use of $B$. This implies that $h$ is \BLR$(B)$. Let $\psi^{A}$ be any function partial computable in $A$. Let $g$ be defined by \bc   $g(n)=\psi^{A}(n)+1$ if $\psi^A(n)\downarrow$,  $g(n)=0$ otherwise. \ec
 Letting $f$ be a computable function with $\psi^A(n)\simeq J(f(n))$ for all $n$, we can use the $B$-computable approximation to $h$  with a computably bounded number of changes  to get such an approximation to $g$. So $g$ is \BLR$(B)$. By Lemma 2.5 of Cole and Simpson \cite{CS}, it follows that \BLR$(A)\subseteq$\BLR$(B)$. 
\end{proof}

\begin{cor}  For c.e.\ sets $A, B$ we have $A \le_{JT} B \leftrightarrow A \le_{BLR} B$. \end{cor} 

\begin{cor}\label{coincides}  JT-hardness coincides with \BLR-hardness: for all $B$, \\ $\ES'\le_{JT}B \leftrightarrow \ES'\le_{BLR} B$. \end{cor}

By Corollary \ref{coincides} and Theorem \ref{consider}((3)$\Implies$(4)),  \BLR-hardness implies superhighness. Cole and Simpson asked \cite[Remark 6.21]{CS} whether conversely superhighness implies   \BLR-hardness. Our negative answer is immediate from Corollary \ref{coincides} and Theorem \ref{consider}((4)$\not\Implies$(3)).

\section{Superhighness,   randomness, and $K$-triviality}

We study the class $\Shigh^\Diamond$ of c.e.\ sets that are Turing below all ML-random superhigh sets.  First we show that this class contains a promptly simple set.  

For background on diagonally non-computable functions and sets of PA degree see \cite[Ch 4]{Nies:book}. Let $\lambda$ denote the usual fair-coin Lebesgue measure on $2^\N$; a null class is a set $\mathcal S\subseteq 2^\N$ with $\lambda(S)=0$.
 
\begin{fact}[Jockusch and Soare \cite{JS}]  \label{ex: PA small}  The sets of PA degree  form a null class. \end{fact}

\begin{proof} 
 Otherwise by the zero-one law  the class is conull. So  by the Lebesgue Density Theorem   there is a Turing functional $\Phi$  such that $\Phi^X(w) \in \twoset$ if defined,  and $$\{Z\colon \, \Phi^Z \ttext{is total and diagonally non-computable}\}$$ has measure at least $3/4$. 
 
 Let the partial computable function $f$ be defined by: $f(n)$ is the value $i\in\{0,1\}$ such that for the smallest possible stage $s$, we observe by stage $s$ that $\Phi^Z(n)=i$ for a set of $Z$s of measure strictly more than $1/4$. For each $n$, such an $i$ and stage $s$ must exist. Indeed, if for some $n$ and both $i\in\{0,1\}$ there is no such $s$, then $\Phi^Z(n)$ is defined for a set of $Z$s of measure at most $\frac{1}{4}+\frac{1}{4}=\frac{1}{2}\not\ge\frac{3}{4}$, which is a contradiction. Moreover, we cannot have $f(n)=J(n)$ for any $n$, because this would imply that there is a set of $Z$s of measure strictly more than $1/4$ for which $\Phi^Z$ is not a total d.n.c.\ function. Thus $f$ is a computable d.n.c.\ function, which is a contradiction.
 \end{proof}

\begin{thm}[Simpson]\label{1}
The class $\Shigh $ of superhigh sets is contained in  a $\Sigma^0_3$ null class.
\end{thm}

\begin{proof} A function $f$ is called diagonally non-computable (d.n.c.) \ relative to $\ES'$ if $\fao x \neg f(x) = J^{\ES'} \! (x)$.   Let~$P$ be the $\PPI(\ES')$ class of  $\twoset$-valued functions that are d.n.c.\ relative to $\ES'$. By Fact~\ref{ex: PA small} relative to $\ES'$, the class $\{Z \colon \, \ex f \leT Z \oplus \ES' \,  [ f\in P]\}$ is null. Then,  since $\GL 1$ is conull, the class 

\bc $\HH = \{Z \colon \, \ex f \ltt Z' \,  [ f\in P]\}$ \ec is also null. This class clearly contains $\Shigh$. 

To show that $\HH$ is $\SI 3$, fix  a $\PI 2$ relation~$R\sub \NN^3$ such that a string  $\sss $ is extended by a member of $P$ iff  $\fao u \exo v  R(\sss, u, v)$. Let $(\Psi_e)\sN e$ be an effective listing of truth-table reduction procedures. It suffices to show that   $\{Z: \Psi_e(Z') \in P\}$ is a $\PI 2$ class. To this end, note that  

\n  \hspace{1cm}  $\Psi_e(Z') \in P \leftrightarrow \fao x \fao t \fao u \ex s> t  \exo v  R(\Psi_e^{Z'} \uhr x [s],u,v)$. 
 \end{proof}
 
A direct construction of a $\Sigma^0_3$ null class containing $\Shigh$ appears in Nies \cite{Nies:nd}. 

\begin{question} Is $\Shigh$ itself a $\Sigma^0_3$ class?
\end{question}

\begin{cor}\label{22} 
There is no superhigh weakly 2-random set.
\end{cor}
\begin{proof}
Let $R$ be a  weakly 2-random set. By definition, $R$ belongs to no $\Pi^0_2$ null class. Since a $\Sigma^0_3$ class is a union of $\Pi^0_2$ classes of no greater measure, $R$ belongs to no $\Sigma^0_3$ null class. By Theorem~\ref{1}, $R$ is not superhigh.   
\end{proof}

To put Corollary \ref{22} into context, recall that the 2-random set $\Om^{\ES'}$ is high, whereas no weakly 3-random set is high (see \cite[8.5.21]{Nies:book}).

\begin{cor}
\label{cor:PSbelow}
There is a promptly simple set Turing below all superhigh ML-random sets.
\end{cor}

\begin{proof}  By a result of Hirschfeldt and Miller (see \cite[Thm. 5.3.15]{Nies:book}), for each null $\SI 3$ class $\mathcal S$ there is a promptly simple set Turing below all   ML-random sets in $\mathcal S$. Apply this to the class $\HH$ from the proof of Theorem \ref{1}.  \end{proof}

Next we show that $\Shigh^\Diamond$    is a proper subclass of the c.e. $K$-trivial sets. Since some superhigh ML-random set is not above $\ES'$,  each set in $\Shigh^\Diamond$  is a base for ML-randomness, and therefore $K$-trivial (for details of this argument, see \cite[Section 5.1]{Nies:book}). It remains to show strictness. In fact in place of the superhigh sets we can consider the possibly smaller class of  sets $Z$ such that $G\ltt Z'$, for some fixed set $G \ge_{tt} \ES''$. Let $\MLR=\{R: R\text{ is ML-random}\}$.
\begin{thm}\label{thm:codeintoshighrds}
 Let $S$ be a $\PPI$ class such that $\ES \subset S \sub \MLR$. Then there is a $K$-trivial c.e.\ set $B$ such that  

\bc $\fao G \ex Z \in S \, [ B \not \leT Z \lland G \ltt Z' ]$. \ec
\end{thm}

\begin{cor} 
There is a $K$-trivial  c.e.\ set $B$ and a superhigh ML-random set $Z$
such that $B\not\le_T Z$. Thus the class of c.e.\ sets Turing below all ML-random superhigh sets  is a proper subclass of the c.e.\ $K$-trivials. 
\end{cor}

\begin{proof}[Proof of Theorem~\ref{thm:codeintoshighrds}] We assume fixed an indexing of all the $\PPI$ classes. Given an index for a $\PPI$ class $P$ we have an  effective approximation $P = \bigcap_t P_t$ where $P_t$ is a clopen set  (\cite[Section 1.8]{Nies:book}).  

To achieve $G \ltt Z'$ we use a variant of \Kuc\ coding.
Given (an index of) a  $\PPI $ class $P$ such that $\ES \subset P \sub \MLR$, we can effectively determine $k\in \NN$ such that $\tp{-k} < \leb P$. In fact $k \le K(i)+ O(1) \le 2 \log i + O(1)$ where $i$ is the index for $P$ (see \cite[3.3.3]{Nies:book}). At stage $t$ let 
\begin{equation} \label{eq: yy} y_{0,t}, y_{1,t}, \end{equation} 
 respectively be the leftmost and rightmost strings $y$ of length $k$ such that $[y] \cap P_t \neq \ES$. Then $y_0$ is left of $y_1$ where $y_a= \lim_t y_{a,t}$.  Note that the number of changes in these approximations is bounded by $2^k$. 

Recall that $(\Phi_e)\sN e$ is an effective listing of the Turing functionals. The following  will be used in a ``dynamic forcing'' construction to ensure that $B \neq \Phi_e^Z$, and to make $B$ $K$-trivial. Let $\cost$ be the standard cost function for building a $K$-trivial set, as defined in \cite[5.3.2]{Nies:book}.  Thus $\cost(x,s) = \sum_{x < w \le s} 2^{-K_s(w)}$.
 
\begin{lemma} \label{changes} Let $Q$ be a $\PPI$ class such that $\ES \subset Q \subset \MLR$. Let $e,m \ge 0$. Then there is a nonempty $\PPI$ class $ P \subset Q$ and $x \in \NN$ such that either 

\bi \item[(a)] $\fa Z \in P \, \neg \Phi_e^Z(x) =0 $, or
\item[(b)] $\ex s \, \cost (x,s) \le \tp{-m}  \lland \fa Z \in P^s_s \, \Phi_{e,s}^Z(x) =0$, \ei
where $(P^t)\sN t$ is an effective sequence of (indices for) $\PPI$ classes such that $P= \lim_t^\text{\rm comp} P^t $ with at most $\tp{m+1}$ changes.   \end{lemma}
The plan is to put $x$ into $B$ in case (b). The change in the approximations $P^t$ is due to changing the candidate $x$ when 
its cost becomes too large.

To prove the lemma, we give a procedure constructing     the required  objects.

\n {\it Procedure $C(Q,e,m)$.}
\n {\it Stage $s$.}
\bi \item[(a)] Choose $x \in \NN^{[e]}$, $x \ge s$.
\item[(b)] If $\cost(x,s) \ge \tp{-m}$, {\sc goto} (a).

\item[(c)] If $\{Z \in Q_s\colon \, \neg \Phi_{e,s}^Z(x) =0\} \neq \ES$ let $P^s = \{Z \in Q\colon \, \neg \Phi_{e}^Z(x) =0\}$ and {\sc goto} (b). (In this case  we keep $x$ out of $B$ and win.) Otherwise let $P^s = Q$ and {\sc goto} (d). (We will put $x$ into  $B$ and win.)

\item[(d)]  {\sc End}. 
\ei
Clearly we choose a new $x$ at most $2^m$ times, so the number of changes of $P^t$ is bounded by $\tp{m+1} $. 

To prove the theorem, we build at each stage $t$ a tree of $\PPI$ classes $P^{\aaa,t}$, where $\aaa \in \strcantor$. The number of changes of  $P^{\aaa,t}$ is bounded computably in~$\aaa$.

\n {\it Stage} $t$. Let $P^{\varnothing, t}= S$.

\bi \item[(i)] If $P=P^{\aaa,t}$ has been defined let, for $b \in \twoset$, \bc $Q^{\aaa b,t} = P^{\aaa,t}\cap [y_{b,t}]$, \ec
where the strings $y_{b,t}$ are as in (\ref{eq: yy}).

\item[(ii)]  If   $Q = Q^{\beta,t}$ is newly   defined let   $e = |\beta|$, let $m$ equal $n_\beta$ (the code number for $\beta$) plus the number of times the index for $Q^\beta$ has changed so far. From now on  define $P^{\beta,t}$ by the procedure $C(Q,e,m)$  in Lemma~\ref{changes}. If it reaches (d), put $x$ into $B$. 
\ei

\n {\bf Claim 1.} {\it  (i) For each $\aaa$ the index $P^{\aaa,t}$ reaches a limit $P^\aaa$. The number of changes is computably bounded in $\aaa$. 

\n  (ii) For each $\beta$ the index $Q^{\beta,t}$ reaches a limit $Q^\beta$. The number of changes is computably bounded in $\beta$. 
}

\n The claim is verified by induction, in the form $P^\aaa \rightarrow Q^{\aaa b} \rightarrow P^{\aaa b}$. This yields a computable definition of the   bound on the number of changes.

Clearly (i) holds when $\aaa = \estring$.

\n {\it Case  $Q^{\aaa b}$:} we can compute by inductive hypothesis an upper bound on the  index for $P^\aaa$, and hence an upper  bound $k_0$ on $k$ such that $\tp {-k} < \leb P^\aaa$. If $N$ bounds the number of changes for $P^\aaa$ then $Q^{\aaa b}$ changes at most $N 2^{k_0}$ times.

\n {\it Case  $P^{\beta}$, $\beta \neq \estring$:} Let $M$ be the bound on the number of changes for   $Q^{\beta}$. Then we always have  $m \le M + n_\beta$ in (ii), so the number of changes for $P^{\beta}$ is at most $M \tp{M + n_\beta +1}$.

\n {\bf Claim 2.} {\it  (i) Let $e = |\beta| >0$.  Then $B \neq \Phi_e(Z)$ for each $Z\in P^\beta$. 
}

\n This is clear, since eventually the procedure in Lemma \ref{changes} has a stable $x$ to diagonalize with.

Given $G$ define $Z\leT \ES' \oplus G$ as follows. For $e >0$ let $\beta = G \uhr e$. Use $\ES'$ to find the final $P^\beta$, and to determine $y_{\beta,b,t}$ ($b \in \twoset$)  for $P= P^\beta$ as  the strings in (\ref{eq: yy}). Let  $y_{\beta,b}= \lim y_{\beta,b,t}$.

 Note that $y_{\gamma} \prec y_\delta$ whenever $\gamma \prec \delta$.  Define $Z$ so that $y_{G(e)} \prec Z$. 

For $G \ltt Z'$ define a function $ f \leT Z$ such that $G(e) = \lim^\text{comp}_s f(e,s)$ (i.e., a computable bounded number of changes).  
Given $e$, to define  $f\uhr e [s]$ search for $t>s$ such that $y_{\aaa,t} \prec Z$ for some $\aaa$ of length $e$, and output $\aaa$. 

\end{proof}

\end{document}